\documentclass{amsart}
\usepackage{graphicx,,amssymb}
\usepackage{amsfonts}
\usepackage{mathrsfs}
\usepackage[all]{xy}
\usepackage{overpic}
\newtheorem{thm}{Theorem}[section]
\newtheorem{lemma}[thm]{Lemma}
\newtheorem{prop}[thm]{Proposition}
\newtheorem{cor}[thm]{Corollary}

\theoremstyle{definition}
\newtheorem{defi}[thm]{Definition}
\newtheorem{example}[thm]{Example}

\theoremstyle{remark}
\newtheorem{remark}[thm]{Remark}
\numberwithin{equation}{section}

\newcommand{\<}{\langle}
\renewcommand{\>}{\rangle}
\newcommand{\D}{\mathcal{D}}      % derived category
\newcommand{\Gammahat}{{\widehat{\Gamma}}}
    
\newcommand{\Hall}{\mathcal{H}}

\newcommand{\Om}{\Omega}

\newcommand{\CC}{\mathbb{C}}       % complex numbers
\newcommand{\Z}{\mathbb{Z}}       % integers
\newcommand{\g}{\mathfrak{g}}
\newcommand{\h}{\mathfrak{h}}
\newcommand{\n}{\mathfrak{n}}

\DeclareMathOperator{\Ind}{Ind}
\DeclareMathOperator{\Hom}{Hom}

\DeclareMathOperator{\Ext}{Ext}
\DeclareMathOperator{\End}{End}

\DeclareMathOperator{\Rep}{Rep}

\begin{document}
\title{Coxeter Elements and Root Bases}
\author{A. Kirillov}
 \address{Department of Mathematics, SUNY at Stony Brook, 
            Stony Brook, NY 11794, USA}
    \email{kirillov@math.sunysb.edu}
    \urladdr{http://www.math.sunysb.edu/\textasciitilde kirillov/} 
\author{J. Thind}
 \address{Department of Mathematics, SUNY at Stony Brook, 
            Stony Brook, NY 11794, USA}
    \email{jthind@math.sunysb.edu}
\maketitle
%%%%%%%%%%%%%%%%%%%%%%%%%%%%%%%%%%%%%%%%%%%%%%%%%%%%%%%%%%

\begin{abstract}
Let $\g$ be a Lie algebra of type $A,D,E$ with fixed Cartan subalgebra $\h$, root system $R$ and Weyl group $W$. We show that a choice of Coxeter element $C \in W$ gives a root basis for $\g$. Moreover, using the results of \cite{kirillov-thind} we show that this root basis gives a purely combinatorial construction of $\g$, where root vectors correspond to vertices of a certain quiver $\Gammahat$, and with respect to this basis the structure constants of the Lie bracket are given by paths in $\Gammahat$. This construction is then related to the constructions of Ringel and Peng and Xiao.
\end{abstract}

%%%%%%%%%%%%%%%%%%%%%%%%%%%%%%%%%%%%%%%%%%%%%%%%%%%%%%%%%%
\section{Introduction}

Let $\Gamma$ be a Dynkin graph of type $A,D,E$. Let $\g$ and $U_{q}( \g )$ be the corresponding Lie algebra and quantum group respectively. By choosing an orientation $\Om$ of $\Gamma$, one obtains a quiver $\overrightarrow{\Gamma} = ( \Gamma , \Om )$. Ringel used the category $\Rep (\overrightarrow{\Gamma} )$ of representations of $\overrightarrow{\Gamma}$ to realise $\n_{+}$ and $U_{q} \n_{+}$ (see \cite{ringel}, \cite{ringel2}). Peng and Xiao then used a related category, $\D^{b} \Rep (\overrightarrow{\Gamma}) /T^{2}$, to realise the whole Lie algebra $\g$. The drawback of these constructions is the necessity of choosing an orientation of the Dynkin diagram. 

Motivated by these results and the ideas of Ocneanu \cite{ocneanu}, the main goal of this paper is to use a Coxeter element, and the results in \cite{kirillov-thind}, to construct a root basis in the Lie algebra $\g$ and to determine the structure constants of the Lie bracket in purely combinatorial terms. 

In \cite{kirillov-thind} it was shown a choice of Coxeter element gives a bijection between $R$ and a certain quiver $\Gammahat$, which identifies roots in $R$ and vertices in $\Gammahat$. This bijection then identifies vertices in $\Gammahat$ with basis vectors $E_{\alpha}$. Using this identification and choice of basis, it is possible to determine the structure constants of the Lie bracket from paths in $\Gammahat$. Thus it is possible to realise the Lie algebra $\g$ completely in terms of the quiver $\Gammahat$. This construction is then independent of any choice of orientation of $\Gamma$ or choice of simple roots.

The case of $U_{q} (\g)$ for $q \neq 1$, is also of interest. However, a full analysis is the subject of ongoing research.

The main result will now be stated. The proof of this theorem will be left to Section~\ref{s:rootvectors} and Section~\ref{s:ringelhall}. In Section~\ref{s:ringelhall} this construction will be related to the constructions of Ringel and Peng-Xiao.

\begin{thm}\label{t:main}
Let $\g$ be a Lie algebra of type $A,D,E$ with fixed Cartan subalgebra $\h$. This gives a root system $R$ with Weyl group $W$. Fix a Coxeter element $C\in W$.
\begin{enumerate}
\item The choice of a Coxeter element $C$ gives a root basis $\{ E_{\alpha} \}_{\alpha \in R}$ for $\g$.
\item Let $\< \cdot , \cdot \>$ be the de-symmetrization of the bilinear form $( \cdot , \cdot )$ given by $$\< x,y \> = ( (1-C)^{-1} x, y).$$ Then the Lie bracket is given by 
$$ [E_{\alpha} , E_{\beta} ] =\begin{cases}

(-1)^{ \< \alpha , \beta \>} E_{\alpha + \beta} &\text{ for } \alpha + \beta \in R \\
0 &\text{ for } \alpha + \beta \not \in R \text{ and } \alpha \neq -\beta
\end{cases}.$$
\end{enumerate}
\end{thm}

%%%%%%%%%%%%%%%%%%%%%%%%%%%%%%%%%%%%%%%%%%%%%%%%%%%%%%%%%%
\section{Preliminaries}

\subsection{Notation} Let $\g$ be a simple Lie algebra of type $A,D,E$, and let $\h$ be a fixed Cartan subalgebra. Denote by $R$ the root system, $W$ the Weyl group, and $\Gamma$ the Dynkin diagram associated to the pair $(\g , \h)$. Thus $\Gamma$ is a Dynkin diagram of type $A,D,E$.

Let $\Pi = \{ \alpha_{i} \}_{i\in \Gamma}$ denote a set of simple roots. 

Since the Weyl group acts simply-transitively on sets of simple roots, there is a unique element which takes $\Pi$ to $-\Pi$. This element is called the longest element and denoted by $w_{0}$.

For $i \in \Gamma$ define $i \ \check{}$ by $- \alpha_{i \ \check{}}= w_{0} ( \alpha_{i} )$, where $w_{0} \in W$ is the longest element.

A set of simple roots $\Pi$ is compatible with a Coxeter element $C \in W$ if there is a reduced expression $C= s_{i_{1}} s_{i_{2}} \cdots s_{i_{r}}$, where each simple reflection appears exactly once. In other words, $\Pi$ is compatible with $C$ if $l^{\Pi} (C) = r$, where $l^{\Pi}$ is the length of a reduced expression in terms of the simple reflections $s_{i}$.

Let $U_{q} (\g)$ be the corresponding quantum group. It is generated by elements $E_{i} , F_{i} , K_{i}^{\pm 1}$, where $i\in \Gamma$. In particular, for $q=1$ this gives the universal enveloping algebra $U(\g)$ of $\g$.

\subsection{The Quiver $\Gammahat$}
Given the Dynkin diagram $\Gamma$ of type $A,D,E$ with Coxeter number $h$, construct a quiver $\Gammahat \subset \Gamma \times \Z_{2h}$ as follows:
\begin{enumerate}
\item Choose a ``parity" function $p:\Gamma \to \Z_{2}$, so that $p(i) = p(j)+1$ for $i,j$ connected in $\Gamma$.
\item Using $p$, define the vertex set of $\Gammahat$ to be $\Gammahat_{0} = \{ (i,n) | \ p(i) + n \equiv 0 \ \ \text{mod2} \}$.
\item The arrows are given by $(i,n) \to (j,n+1)$ for $i,j$ connected in $\Gamma$.
\item Define a ``twist" map $\tau :\Gammahat \to \Gammahat$ by $\tau (i,n) = (i,n+2)$.
\end{enumerate}

\begin{figure}[t]
        \includegraphics[height=2.50in]{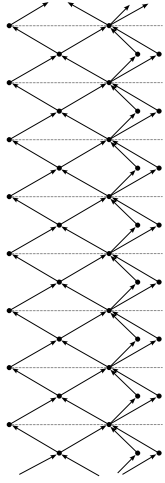}
        \caption{The quiver $\Gammahat$ for graph $\Gamma = D_{5}$. Recall that $\Gammahat$ is periodic, so the arrows leaving the top level are the same as the incoming arrows at the bottom level.}\label{f:Ihat-D5}
    \end{figure}

\begin{example}
For the graph $ \Gamma = D_{5}$ the quiver $\Gammahat$ is shown in
    Figure~\ref{f:Ihat-D5}. Note that this is the Auslander-Reiten quiver of the category $\D^{b} ( \overrightarrow{\Gamma} ) / T^{2}$ for any choice of orientation $\Om$. Here the $\Z_{2h}$ direction is vertical and the translation acts vertically, while in most of the literature the $\Z_{2h}$ direction is horizontal and the translation acts horizontally to the right.
\end{example}

A function $h: \Gamma \to \Z_{2h}$ such that $h(i) = h(j) \pm 1$ for $i,j$ connected in $\Gamma$ will be called a ``height function". Note that such a map defines an orientation on $\Gamma$ by $i \to j$ if $i,j$ are connected and $h(j)=h(i)+1$. This orientation will be denoted by $\Om_{h}$. A height function $h$ also gives an embedding of the quiver $(\Gamma , \Om_{h})$ in $\Gammahat$, given by $i \mapsto (i,h(i))$. The image of such an embedding is called a ``slice" and is denoted by $\Gamma_{h}$.

For a height function $h$, if $i \in \Gamma$ is a sink or source for $\Om_{h}$ define a new height function $s_{i} h$ by 

$$s_{i} h(j) = \begin{cases}
h(i) \pm 2 \ \text{ if } j=i \ \text{ where the sign is + for i a source, - for i a sink} \\
h(j) \ \text{ if } j \neq i
\end{cases}$$
The orientation determined by $s_{i} h$ is denoted by $s_{i} \Om$ and is obtain by reversing all arrows at $i$.

Define a function  $\< \cdot ,\cdot\>_{\Gammahat} \colon \Gammahat \times \Gammahat\to\Z$
    by setting 

    \begin{align*}
    \< (i,n),(j,n) \>_{\Gammahat} &= \delta_{ij}, \\
          \< (i,n),(j,n+1) \>_{\Gammahat} &=\text{the number of paths}
                (i,n)\to \cdots \to  (j,n+1)\\
            &= \text{the number of edges between } i,j \text{ in } \Gamma.
     \end{align*}

    Then for any $q=(k,m)\in \Gammahat$ use the relation 
        $$
        \< q, (i,n) \>_{\Gammahat} 
          -\sum_{j-i} \< q, (j,n+1) \>_{\Gammahat} 
        +\< q, (i,n+2) \>_{\Gammahat} = 0
        $$ for $i,j$ connected in $\Gamma$, to extend the definition.
        
It was shown in \cite{kirillov-thind} (Proposition 7.4) that this function is well-defined.

Given a Coxeter element $C\in W$, it was shown in \cite{kirillov-thind} that there is a bijection $R \to \Gammahat$ with the following properties:
\begin{enumerate}
\item It identifies the Coxeter element $C$ with the ``twist" $\tau: \Gammahat \to \Gammahat$.
\item It gives a bijection between simple systems $\Pi$, compatible with $C$, and height functions $h: \Gamma \to \Gammahat$. 
\item For each height function $h$ one obtains an explicit description of the corresponding positive roots and negative roots as disjoint connected subquivers of $\Gammahat$, as well as a reduced expression for the longest element $w_{0}$ in the Weyl group. The reduced expression for $w_{0}$ is given as a sequence of source to sink reflections taking the slice $\Gamma_{h^{\Pi}}$ to the slice $\Gamma_{h^{-\Pi}}$.
\item There is a de-symmetrization of the inner product on $R$, denoted by $\< \cdot , \cdot \>$ which is analogous to the Euler form in quiver theory. Moreover, under the bijection $\Phi$, this form is identified with $\< \cdot , \cdot \>_{\Gammahat}$ in $\Gammahat$.
\end{enumerate}

\begin{figure}[ht]
\includegraphics[height=2.50in]{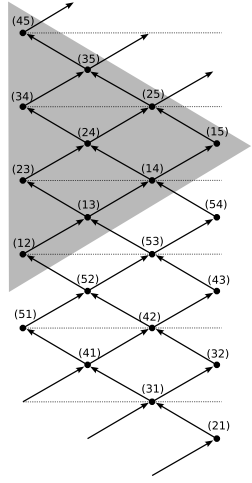}
\caption{The bijection $\Phi : R \to \Gammahat$ for $\Gamma = A_{4}$. For each vertex in $\Gammahat$ the corresponding root $\alpha \in R$ is shown. The notation $(ij)$ stands for $e_{i} - e_{j}$. The set of positive roots corresponding to $\Pi$ is shaded. Recall that $\Gammahat$ is periodic, so that arrows leaving the top level are identified with the incoming arrows on the bottom level.}\label{f:phiA4}
\end{figure}

\begin{example}\label{e:a41}
For $\Gamma = A_{4}$ with $\Pi = \{ e_{1} - e_{2} , e_{2} - e_{3} , e_{3}-e_{4} , e_{4}-e_{5} \}$ and $C=s_{1} s_{2} s_{3} s_{4}$ the bijection $R \to \Gammahat$ is given in Figure~\ref{f:phiA4}.
\end{example}

%%%%%%%%%%%%%%%%%%%%%%%%%%%%%%%%%%%%%%%%%%%%%%%%%%%%%%%%%%%
\section{Braid Group Action}
In this section the definition and relevant results of the braid group operators as defined in \cite{jantzen} are reviewed. For more details see \cite{jantzen}, or \cite{lusztig}.
 
Fix a system of simple roots $\Pi$. Let $E_{i}, F_{i}, K_{i}^{\pm 1}$ denote the corresponding generators of $U_{q} (\g)$. \\
For simple roots $\alpha_{i}$ define operators $T_{i}, T_{i}^{\prime}$ on any finite dimensional module $V$ by setting for $v\in V_{\lambda}$:
$$T_{i}(v) = \sum_{a,b,c \geq 0;-a+c-b=m} (-1)^{b} q^{b-ac} \frac{E_{i}^{(a)}}{[a]!} \frac{F_{i}^{(b)}}{[b]!} \frac{E_{i}^{(c)}}{[c]!} v$$
$$T_{i}^{\prime} (v) = \sum_{a,b,c \geq 0;-a+c-b=m} (-1)^{b}q^{ac-b} \frac{E_{i}^{(a)}}{[a]!}  \frac{F_{i}^{(b)}}{[b]!} \frac{E_{i}^{(c)}}{[c]!} v$$
with $m=(\lambda, \alpha_{i} \check{} \ )$.

Then there are unique automorphisms of $U_{q} (\g)$, also denoted by $T_{i}, T_{i}^{\prime}$ so that for any $u\in U_{q} (\g)$ and $v\in V$ we have $T_{i} (uv) = T_{i}(u) T_{i}(v)$. The operator $T_{i}$ acts on weights by the reflection $s_{i}$.

The automorphisms $T_{i}$ satisfy the braid relations:
$$ T_{i} T_{j} = T_{j}T_{i} \ \ \ \text{for} \ \ (\alpha_{i} , \alpha_{j})=0$$
$$ T_{i} T_{j} T_{i} = T_{j}T_{i}T_{j} \ \ \text{for} \ \ (\alpha_{i} , \alpha_{j} )=-1$$
For the automorphism $T_{i}$ there are the following formulae:

\begin{align*}
&T_{i} E_{i} = -F_{i}K_{i} & \\
&T_{i} F_{i} = -K_{i}^{-1} E_{i} & \\
&T_{i} E_{j} = E_{j} \ \ \text{for} \ \ (\alpha_{i} , \alpha_{j}) = 0 & \\
&T_{i} E_{j} = E_{i}E_{j} - q^{-1}E_{j}E_{i}  \ \ \text{for} \ \ (\alpha_{i} , \alpha_{j} ) =-1 & \\
&T_{i} F_{j} = F_{j} \ \ \text{for} \ \ (\alpha_{i} , \alpha_{j}) = 0 & \\
&T_{i} F_{j} = F_{i}F_{j} - q^{-1}F_{j}F_{i}  \ \ \text{for}  \ \ (\alpha_{i} , \alpha_{j} ) =-1 & \\
\end{align*}

In fact, there are automorphisms $T_{\alpha}$ for any root $\alpha$. As above, define $T_{\alpha}$ on a module $V$ by setting for $v\in V_{\lambda}$:
$$T_{i}(v) = \sum_{a,b,c \geq 0;-a+c-b=m} (-1)^{b} q^{b-ac} \frac{E_{\alpha}^{(a)}}{[a]!} \frac{F_{\alpha}^{(b)}}{[b]!} \frac{E_{\alpha}^{(c)}}{[c]!} v$$
$$T_{i}^{\prime} (v) = \sum_{a,b,c \geq 0;-a+c-b=m} (-1)^{b}q^{ac-b} \frac{E_{\alpha}^{(a)}}{[a]!}  \frac{F_{\alpha}^{(b)}}{[b]!} \frac{E_{\alpha}^{(c)}}{[c]!} v$$
where $E_{\alpha} , F_{\alpha} \in U_{q} (\g)$ satisfy the $U_{q} (sl_{2})$ relations and $m=(\lambda, \alpha_{i} \check{} \ )$.

\begin{lemma} Let $\Phi$ be an automorphism of $U_{q} (\g)$ such that $E_{\alpha} = \Phi (E_{i})$ and $F_{\alpha} = \Phi (F_{i})$. Then $T_{\alpha} = \Phi T_{i} \Phi^{-1}$. 
\end{lemma}

The automorphisms $T_{\alpha}$ satisfy relations similar to those of the $T_{i}$:

\begin{align*}
&T_{\alpha} E_{\alpha} = -F_{\alpha}K_{\alpha} & \\
&T_{\alpha} F_{\alpha} = -K_{\alpha}^{-1} E_{\alpha} & \\
&T_{\alpha} E_{\beta} = E_{\beta} \ \ \text{for} \ \ (\alpha , \beta) = 0 & \\
&T_{\alpha} E_{\beta} = E_{\alpha}E_{\beta} - q^{-1}E_{\beta}E_{\alpha}  \ \ \text{for} \ \ (\alpha , \beta ) =-1 & \\
&T_{\alpha} F_{\beta} = F_{\beta} \ \ \text{for} \ \ (\alpha , \beta ) = 0 & \\
&T_{\alpha} F_{\beta} = F_{\alpha}F_{\beta} - q^{-1}F_{\beta}F_{\alpha}  \ \ \text{for}  \ \ (\alpha , \beta ) =-1 & \\
\end{align*}

Since the operators $T_{i}$ satisfy the braid relations it is possible to define an operator $T_{w}$ for any $w\in W$. For any reduced expression $w= s_{i_{1}} s_{i_{2}} \cdots s_{i_{k}}$ for $w\in W$ define $T_{w} = T_{i_{1}} T_{i_{2}} \cdots T_{i_{k}}$. 

The following Lemma will be useful. It can be found in \cite{jantzen}, Proposition 8.20.

\begin{lemma}
If $w \in W$ is such that $w(\alpha_{i}) \in R_{+}$, then $T_{w} (E_{i}) \in U_{q}^{+}$. If, in addition, $w(\alpha_{i}) = \alpha_{j}$, then $T_{w} (E_{i}) = E_{j}$.
\end{lemma}

For the case to be considered in the following sections, this result gives the following important Corollary.

\begin{cor}\label{c:jantzen}
Let $w_{0} \in W$ be the longest element. Then $T_{w_{0}} (E_{i \ \check{}}) = T_{i} E_{i} = -F_{i}K_{i}$.
\end{cor}

\begin{proof}
Let $w_{0}= s_{i}w$ be a reduced expression for $w_{0}$, so that $T_{w_{0}} = T_{i} T_{w}$. Then since $$w(\alpha_{i \ \check{}}) = s_{i}w_{0} (\alpha_{i \ \check{}}) = s_{i} (-\alpha_{i}) = \alpha_{i}$$ the Lemma gives $T_{w} (E_{i \ \check{}}) = E_{i}$, and the result follows by applying $T_{i}$.
\end{proof}

%%%%%%%%%%%%%%%%%%%%%%%%%%%%%%%%%%%%%%%%%%%%%%%%%%%%%%%%%%%
\section{Longest Element and Construction of Root Vectors}\label{s:rootvectors}

Let $\Pi$ be a simple system, and let $R= R_{+} \cup R_{-}$ be the corresponding polarization. Let $w_{0}$ be the longest element. A reduced expression $w_{0} = s_{i_{1}} s_{i_{2}} \cdots s_{i_{l}}$ is said to be {\em adapted} to an orientation $\Om$ of $\Gamma$ if $i_{k}$ is a source for $s_{i_{1}} \cdots s_{i_{k-1}} \Om$. In particular $i_{1}$ is a source of $\Om$.

\begin{lemma}\label{l:redexp}
Given any orientation $\Om$, there is a reduced expression adapted to $\Om$, and moreover, any two expressions adapted to $\Om$ are related by $s_{i}s_{j} = s_{j}s_{i}$ with $n_{ij}=0$.
\end{lemma}
\begin{proof}
Recall that any height function $h$ determines an orientation $\Om_{h}$ and that for any orientation $\Om$ there is a choice of $h$ so that $\Om = \Om_{h}$. So take some $h$ corresponding to $\Om$. Note that any reduced expression adapted to $\Om$ gives a sequence of source to sink moves taking the slice $\Gamma_{h}$ to the slice $\Gamma_{-h}$ where $\Gamma_{-h}$ is the slice corresponding to the simple roots $-\Pi$.

Let $w_{0} = s_{i_{1}} \cdots s_{i_{l}}$ and $w_{0} = s_{i_{1}^{\prime}} \cdots s_{i_{l}^{\prime}}$ be two different reduced expressions adapted to $\Om$. Let $k$ be the first index where they differ. Write $i_{k} = i$ and $i_{k}^{\prime} = j$ to simplify notation. Then there are reduced expressions
$$w_{0}  = ws_{i} w_{1}s_{j}w_{2}$$
$$w_{0} = ws_{j} w_{1}^{\prime} s_{j} w_{2}^{\prime}$$ 
where $s_{j}$ does not appear in $w_{1}$ and $s_{i}$ does not appear in $w_{1}^{\prime}$. Thus $i,j$ are both sources for $w\Om$ and hence $n_{ij} = 0$. Note that since $w_{1}$ is obtained as a sequence of source to sink reflections, and since $s_{j}$ does not appear in $w_{1}$, $j$ remains a source during this process. Hence if $s_{k}$ appears in $w_{1}$ then $k$ is not adjacent to $j$, so that $n_{jk} = 0$. Thus $w_{1}s_{j} = s_{j}w_{1}$. which gives:

\begin{align*}
w_{0}  &= ws_{i} w_{1}s_{j}w_{2} \\
&= ws_{i}s_{j}w_{1}w_{2}\\
&= ws_{j}s_{i}w_{1}w_{2}
\end{align*}
So it is possible to make the two reduced expressions agree at the index $k$ using only the relation $s_{i}s_{k} = s_{k}s_{i}$ for $n_{ik} = 0$. Continuing in this fashion it is possible to make the expressions agree at every index using only this relation. 

\end{proof}

It is well known that a reduced expression $w_{0} = s_{i_{1}} \cdots s_{i_{l}}$, adapted to $\Om$, gives an ordering of the positive roots $R= \{ \gamma_{1} , \ldots , \gamma_{l} \}$ by setting $\gamma_{k} = s_{i_{1}} \cdots s_{i_{k-1}} ( \alpha_{k} )$. Such an expression also gives roots vectors $E_{\alpha} ,F_{\alpha}$ for $\alpha \in R_{+}$ as follows:

\begin{equation}\label{e:stdrootvec}
E_{\gamma_{k}} = T_{i_{1}} \cdots T_{i_{k-1}} (E_{i_{k}}) 
\end{equation}
\begin{equation}
F_{\gamma_{k}} = T_{i_{1}} \cdots T_{i_{k-1}} (F_{i_{k}})
\end{equation}

Note that since the $T_{i}$ satisfy the braid relation, Lemma~\ref{l:redexp} implies that the root vectors defined this way do not depend on the choice of reduced expression adapted to $\Om$.

Note that if $\gamma_{k} = s_{i_{1}} \cdots s_{i_{k-1}} \alpha_{i_{k}}$ then $\gamma_{k} = s_{\gamma_{k-1}} \cdots s_{\gamma_{1}} \alpha_{i_{k}}$, and the longest element can be expressed as $w_{0} = s_{\gamma_{l}} \cdots s_{\gamma_{1}}$.

Since $T_{\gamma_{k}} = (T_{i_{1}} \cdots T_{i_{k-1}} )T_{i_{k}} (T_{i_{1}} \cdots T_{i_{k-1}})^{-1}$, then as for reflections,
$$T_{i_{1}} \cdots T_{i_{k}} = T_{\gamma_{k}} \cdots T_{\gamma_{1}}$$
so the root vectors $E_{\gamma_{k}}$ given in Equation~\ref{e:stdrootvec} can be expressed as
\begin{equation}\label{e:rootvecgamma}
E_{\gamma_{k}} = T_{\gamma_{k-1}} \cdots T_{\gamma_{1}} E_{i_{k}}
\end{equation}

\begin{defi}
Let $\Pi$ be a set of simple roots, and $\Om$ be an orientation of $\Gamma$ and let $w_{0} = s_{i_{1}} \cdots s_{i_{l}}$ be a reduced expression adapted to $\Om$. A root basis $\{ E_{\alpha} \}_{\alpha \in R}$ is said to be adapted to the pair $(\Pi , \Om)$ if for $\alpha \in R_{+}$ the vector $E_{\alpha}$ is given by Equation~\ref{e:stdrootvec}, or equivalently by Equation~\ref{e:rootvecgamma}.
\end{defi}

\subsection{Change of Orientation} For a reduced expression $w_{0} = s_{i_{1}} s_{i_{2}} \cdots s_{i_{l}}$, adapted to $\Om$, define a new reduced expression $w_{0} =s_{i_{2}} \cdots s_{i_{l}} s_{i_{1}  \check{}} $ which is adapted to $s_{1} \Om$. Then this gives a new enumeration of positive roots $\{ \gamma_{1}^{\prime}, \ldots \gamma_{l}^{\prime} \}$, and a new collection of root vectors:
$$ \gamma_{1}^{\prime} = s_{i_{i}} (\gamma_{i_{2}}) , \gamma_{2}^{\prime} = s_{i_{1}} (\gamma_{i_{3}}), \ldots, \gamma_{l} = \alpha_{i_{1} \check{}}$$
\begin{equation}\label{e:newrootvec}
E_{\gamma^{\prime}} = T_{i}^{-1} (E_{s_{i} \gamma}) \ \ \ \ \text{for} \ \ \  \gamma \neq \alpha_{i_{1}} 
\end{equation}

\subsection{Coxeter Element} Now consider the case where there is a fixed Coxeter element $C \in W$ and hence an identification $R \to \Gammahat$ as in \cite{kirillov-thind}. In this case a choice of height function $h$ is identified with a set of simple roots $\Pi$ compatible with $C$, and hence determines a polarization $R= R_{+}^{h} \cup R_{-}^{h}$. A height function also determines a reduced expression for $w_{0}$ adapted to the orientation $\Om_{h}$. This expression is obtained from $\Gammahat$ as a sequence of source to sink reflections which take the slice $\Gamma_{h^{\Pi}}$ to the slice $\Gamma_{h^{-\Pi}}$. 

Using this reduced expression, there is an associated ordering of the positive roots which gives a completion of the partial order given by paths in $\Gammahat$. Note that the completion depends on the reduced expression. 

Now choose a height function $h$. Then using the reduced expression for $w_{0}$ obtained above, it is possible to define a collection of root vectors $E_{\alpha}$ for $\alpha \in R_{+}^{h}$ using Equation~\ref{e:stdrootvec}.

Under the identification $R \to \Gammahat$ suppose that $\alpha = (i,n)$, then $C\alpha = (i,n+2)$. For $j$ connected to $i$, denote by $\gamma_{j}$ the root corresponding to vertex $(j, n+1)$. The collection of roots $\{ \alpha, \gamma_{j} , C\alpha \}$ is said to satisfy the fundamental relation in $\Gammahat$. Such a collection is depicted in Figure~\ref{f:fundrel}.

\begin{figure}[ht]
\includegraphics[height=1.50in]{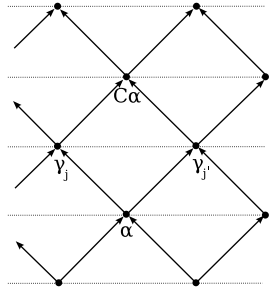}
\caption{A collection of roots $\alpha , \gamma_{j} , C\alpha \in \Gammahat$ satisfying the fundamental relation.}\label{f:fundrel}
\end{figure}

\begin{lemma}
Let  $\alpha, \gamma_{j} ,C (\alpha) \in R_{+}^{h}$ satisfy the fundamental relation in $\Gammahat$. Then the corresponding root vectors satisfy:

\begin{equation}\label{e:rootvec}
E_{C(\alpha)} = (\prod_{j-i} T_{\gamma_{j}} ) T_{\alpha} (E_{\alpha})
\end{equation}
\end{lemma}

\begin{proof}
Let $h$ be a fixed height function and let $\Pi = \{ \alpha_{1} , \ldots , \alpha_{r} \}$ denote the corresponding set of simple roots and $s_{i}$ the corresponding simple reflections.

Let $\alpha = (i,n) , \gamma_{j} = (j,n+1) , C\alpha = (i,n+2)$ and $$w_{0} = ws_{i} ( \prod_{j-i} s_{j}) s_{i} w^{\prime}$$ a reduced expression adapted to $\Om_{h}$.

Then
\begin{align*}
&E_{\alpha} = T_{w} E_{\alpha_{i}}  \text{ and }  T_{\alpha} = T_{w} T_{\alpha_{i}} T_{w}^{-1} \\
&E_{\gamma_{j}} = T_{w} T_{\alpha_{j}} E_{\alpha_{j}} \text{ and } T_{\gamma_{j}} = T_{w} T_{\alpha_{i}} T_{\alpha_{j}} T_{\alpha_{i}}^{-1} T_{w}^{-1} \\
&E_{C\alpha} = T_{w} T_{\alpha_{i}} (\prod_{j-i} T_{\alpha_{j}}) E_{\alpha_{i}}
\end{align*}
where the product $\prod_{j-i}$ is taken over all $j$ connected to $i$ in $\Gamma$.

On the other hand, using the first two formulae, and comparing with the third one obtains:
\begin{align*}
(\prod_{j-i} T_{\gamma_{j}}) T_{\alpha} E_{\alpha} &= \prod_{j-i}( T_{w} T_{\alpha_{i}} T_{\alpha_{j}} T_{\alpha_{i}}^{-1} T_{w}^{-1} ) T_{w} T_{\alpha_{i}} T_{w}^{-1} T_{w} E_{\alpha_{i}} \\
&= T_{w} T_{\alpha_{i}} ( \prod_{j-i} T_{\alpha_{j}}) T_{\alpha_{i}}^{-1} T_{w}^{-1} T_{w} T_{\alpha_{i}} T_{w}^{-1} T_{w} E_{\alpha_{i}} \\
&=  T_{w} T_{\alpha_{i}} ( \prod_{j-i} T_{\alpha_{j}}) E_{\alpha_{i}} \\
& = E_{C\alpha}\\
\end{align*}

\end{proof}

Now fix a height function $h$, and hence a choice of compatible simple roots $\Pi_{h}$, an orientation $\Om_{h}$, a reduced expression $w_{0}$ and a slice $\Gamma_{h} \subset \Gammahat$. Define a root basis as follows:
\begin{enumerate}
\item For $\beta \in \Gamma_{h}$ choose $E_{\beta} \in \g_{\beta}$.
\item Since any root is of the form $\alpha = C^{k} \beta$ for some $k$ and some $\beta$, define $E_{\alpha}$ inductively using Equation~\ref{e:rootvec}, beginning with $E_{C\beta}$ for $\beta$ a source in $\Gamma_{h}$.
\end{enumerate}

Note that for $C^{h} \beta = \beta$ this procedure produces another root vector $E_{\beta}^{\prime} \in \g_{\beta}$.

\begin{prop}\label{p:basis} Let $E_{\alpha}, E_{\alpha}^{\prime}$ be the root vectors defined above.
\begin{enumerate}
\item For $q=1$ $E_{\alpha}^{\prime} = E_{\alpha}$, so this procedure produces a consistent root basis in $\g$.
\item For $q \neq 1$ $E_{\alpha}^{\prime} = K_{\alpha}^{-1} E_{\alpha} K_{\alpha}$.
\end{enumerate}
\end{prop}

\begin{cor}\label{c:quantum}
Let $U_{q} \g$ denote the corresponding quantum group. For $q \neq 1$ there is a $\Z$-torsor of vectors $\{ E_{\alpha}^{k} \}$ for each root $\alpha$ that are related by $E_{\alpha}^{k+n} = K_{\alpha}^{n}E_{\alpha}^{k} K_{\alpha}^{-n}$.
\end{cor}

\begin{proof} To simplify notation, set $E_{\alpha_{i}} = E_{i} , F_{\alpha_{i}} = F_{i} , K_{\alpha_{i}} = K_{i}$. Then using Corollary~\ref{c:jantzen} one obtains:
\begin{align*}
E_{i}^{\prime} &= T_{w_{0}} (T_{w_{0}} E_{i}) \\
&= T_{w_{0}}( -F_{i \ \check{}} \ K_{i \ \check{}} ) \\
&= -(T_{w_{0}} F_{i \ \check{ }})(T_{w_{0}} K_{i \ \check{}}) \\
&= -(-K_{i}^{-1} E_{i}) (K_{i}) \\
&= K_{i}^{-1} E_{i} K_{i}
\end{align*}
This proves the second part, and to get the first part set $q=1$ so that $K_{i} = 1$.
\end{proof}

\begin{thm}\label{p:adapted}
Let $h$ be any height function and denote the associated simple roots and orientation by $\Pi_{h}$ and $\Om_{h}$ respectively.
\begin{enumerate} 
\item The root basis defined above is adapted to the pair $( \Pi_{h} , \Om_{h} )$.
\item For this choice of root basis the Lie bracket is given by:
$$ [E_{\alpha} , E_{\beta} ] =\begin{cases}

(-1)^{ \< \alpha , \beta \>} E_{\alpha + \beta} &\text{ for } \alpha + \beta \in R \\
0 &\text{ for } \alpha + \beta \not \in R \text{ and } \alpha \neq -\beta
\end{cases}$$
\end{enumerate}
\end{thm}

\begin{proof}
Let $h$ be the height function used to construct the root basis $\{ E_{\alpha} \}$. By construction this basis is adapted to the pair $( \Pi_{h} , \Om_{h} )$. So it is enough to check that if $\{ E_{\alpha} \}$ is adapted to $( \Pi , \Om )$, and $i$ is a source for $\Om$, then $\{ E_{\alpha} \}$ is also adapted to $( s_{i} \Pi , s_{i} \Om )$.

Suppose that $\{ E_{\alpha} \}$ is adapted to $( \Pi , \Om )$ and that $i$ is a source. Let $\Pi = \{ \alpha_{1} , \ldots , \alpha_{r} \}$. Then since $i$ is a source and $w_{0}$ is adapted to $\Om$, the corresponding reduced expression for the longest element has the form $w_{0} = s_{i} s_{i_{2}} \cdots s_{i_{l}}$. By writing $\gamma_{k} = s_{i_{1}} \cdots s_{i_{k-1}} \alpha_{i_{k}}$, the longest element can be reexpressed as $w_{0} = s_{\gamma_{l}} \cdots s_{\gamma_{2}} s_{\alpha_{i}}$. (Note that since $i$ is a source, $\gamma_{1} = \alpha_{i}$.)

Then since $\{ E_{\alpha} \}$ is adapted it is possible to write 
$$E_{\gamma_{k}} = T_{\gamma_{k-1}} \cdots T_{\gamma_{2}} T_{\alpha_{i}} E_{\alpha_{k}}.$$

Now consider the pair $( s_{i} \Pi , s_{i} \Om )$. Denote the simple roots by $\alpha_{j}^{\prime} = s_{i} \alpha_{j}$ and the corresponding simple reflections $s_{j}^{\prime} = s_{i} s_{j} s_{i}$. Then the corresponding reduced expression for the longest element is $w_{0} = s_{i_{2}}^{\prime} \cdots s_{i_{l}}^{\prime} s_{i}$ and as before if $\gamma_{k}^{\prime} = s_{i_{2}}^{\prime} \cdots s_{i_{k-1}}^{\prime} (\alpha_{i_{k}}^{\prime})$, then $\gamma_{k}^{\prime} = \gamma_{k+1}$ for $k+1 \neq l$ and $\gamma_{l} = -\alpha_{i}$.

Now, if $k+1 \neq l$ then 
\begin{align*}
E_{\gamma_{k}}^{\prime} &= E_{\gamma_{k+1}} \\
&= T_{\gamma_{k}} \cdots T_{\gamma_{2}} T_{\alpha_{i}} E_{\alpha_{i_{k+1}}} \\
&= T_{\gamma_{k}} \cdots T_{\gamma_{2}} E_{s_{i} \alpha_{i_{k+1}}} \ \ \text{by Equation~\ref{e:newrootvec}} \\
&= T_{\gamma_{k-1}^{\prime}} \cdots T_{\gamma_{1}^{\prime}} E_{\alpha_{k}^{\prime}} 
\end{align*}
so $E_{\gamma_{k}}^{\prime}$ is given by Equation~\ref{e:rootvecgamma}.

If $k+1 = l$ then, 
\begin{align*}
E_{\gamma_{l}}^{\prime} &= E_{-\alpha_{i}} \\
&= T_{w_{0}} (E_{\alpha_{i \ \check{} }}) \\
&= T_{\gamma_{l}} \cdots T_{\gamma_{2}} T_{\alpha_{i}} (E_{\alpha_{i \ \check{} }}) \\
&= T_{\gamma_{l-1}}^{\prime} \cdots T_{\gamma_{1}}^{\prime} E_{s_{i} \alpha_{i} \ \check{}} \\
&= T_{\gamma_{l-1}}^{\prime} \cdots T_{\gamma_{1}}^{\prime} E_{\alpha_{i}^{\prime} \ \check{}} 
\end{align*}
so again $E_{\gamma_{k}}^{\prime}$ is given by Equation~\ref{e:rootvecgamma}. Hence $\{ E_{\alpha} \}$ is adapted to the pair $( s_{i} \Pi , s_{i} \Om )$.

The proof of the second part will follow from Corollary~\ref{c:bracket}.
\end{proof}

Note that Proposition~\ref{p:basis} and Proposition~\ref{p:adapted} prove the main result, Theorem~\ref{t:main}.

Define $T_{C} = T_{\alpha_{i_{1}}} T_{\alpha_{i_{2}}} \cdots T_{\alpha_{i_{r}}}$, for some choice of compatible simple roots $\Pi = \{ \alpha_{1} , \ldots , \alpha_{r} \}$, with $C = s_{i_{1}} s_{i_{2}} \cdots s_{i_{r}}$. Since the $T_{\alpha}$ satisfy the braid relations, the operator does not depend on the choice of compatible simple roots $\Pi$.

\begin{prop}
The root vectors $E_{\alpha}$ satisfy $T_{C} E_{\alpha} = E_{C \alpha}$.
\end{prop}

%%%%%%%%%%%%%%%%%%%%%%%%%%%%%%%%%%%%%%%%%%%%%%%%%%%%%%%%%%%
\section{Ringel-Hall Algebras}\label{s:ringelhall}

In this section Ringel and Peng and Xiao's approaches to constructing the Lie algebra $\g$ from quiver theory is reviewed. This is then related to the construction given in the previous section. For more details on Ringel's construction see \cite{ringel}, \cite{ringel2}, \cite{dengxiao}. For more details on Peng and Xiao's construction see \cite{px} and \cite{px2}.

Let $\Om$ be a fixed orientation of $\Gamma$ and denote by $\overrightarrow{\Gamma} = (\Gamma, \Om )$ the corresponding quiver. Fix $\mathbb{K}$, a finite field of order $p$. Let $\Rep (\overrightarrow{\Gamma} )$ be the category of representations of this quiver over the field $\mathbb{K}$, and denote by $\mathcal{K}$ its Grothendieck group. Denote by $\Ind \subset \mathcal{K}$ the set of classes of indecomposable objects.  Then Gabriel's Theorem gives an identification $\Ind \to R_{+}$ between indecomposable objects and positive roots of the corresponding root system. Moreover, if $\< \cdot , \cdot \>$ is defined on $\mathcal{K}$ by $\< X , Y \> = \dim \Hom (X, Y) - \dim \Ext^{1} (X,Y)$, then the form given by $( X,Y ) = \< X,Y \> + \< Y,X \>$ is identified with the bilinear form on the root lattice. The form $\< \cdot , \cdot \>$ is called the Euler form.

Ringel then constructed an associative algebra $( \Hall_{p} , \ast )$ as follows:

\begin{enumerate}
\item As a vector space $\Hall_{p}$ is spanned by $[M] \in \mathcal{K}$.
\item For objects $M,N,L$ define $F_{M,N}^{L} = | \{ X\subset L | X\simeq M \ \text{and} \ L/X \simeq N \} |$. (Since $\mathbb{K}$ is finite, this number is well-defined.)
\item Define an operation $\ast$ on $\Hall_{p}$ by the formula $ [M] \ast [N] = \sum_{ [L]} F_{M,N}^{L} [L]$. 
\end{enumerate}

The following Theorem summarizes the main results of Ringel.
\begin{thm} Let $( \Hall_{p} , \ast)$ be the algebra defined above.
\begin{enumerate} 
\item For $n = (n_{\alpha}) \in (\Z_{+})^{R_{+}}$ set $M_{n} = \bigoplus n_{\alpha} M_{\alpha}$ where $M_{\alpha}$ is the indecomposable corresponding to root $\alpha$. Then $\{ [M_{n}] \}$ is a PBW-type basis of the algebra $\Hall_{p}$, so that all structure constants $F_{[M] , [N] }$ are in $\Z [p]$. Hence the Hall algebra can be considered with $p$ as a formal parameter. After making the substitution $q=p^{1/2}$, $\Hall_{q}$ can be identified with $U_{q} \n_{+}$.
\item For $q = 1$, this gives an isomorphism $\Psi : U \n_{+} \to \Hall_{1}$ which is given by $E_{\alpha} \mapsto [M_{\alpha}]$, where $M_{\alpha}$ denotes the indecomposable representation of $\overrightarrow{\Gamma}$ corresponding to root $\alpha$. In particular the set $\{ [M_{\alpha} ] \}$ is a root basis for $\n_{+}$.
\item In the case $q=1$, the Lie bracket $[ \cdot , \cdot ]$ is given by $[ M_{\alpha} , M_{\beta} ] = (-1)^{\< M_{\alpha} , M_{\beta} \> } M_{\alpha + \beta}$ for $\alpha + \beta \in R$. Here $\< \cdot , \cdot \>$ is the Euler form.
\end{enumerate}
\end{thm}

The polynomial $F_{M,N}^{L} (p)$ appearing in Part 1 of the Theorem is called the ``Hall polynomial".

As mentioned before, Peng and Xiao extended the results of Ringel to obtain a description for all of $\g$. This construction is briefly recalled here. For a more details see \cite{px} , \cite{px2}.

Peng and Xiao considered the ``root category", $\D = \D^{b} (\overrightarrow{\Gamma}) / T^{2}$, so that indecomposable objects are in bijection with all roots. If $M \in \Rep (\overrightarrow{\Gamma})$ is indecomposable, then considering this as a complex concentrated in degree 0, $M$ is also indecomposable in $\D$. These objects correspond to positive roots, while their translates, $TM$, correspond to negative roots. (Up to isomorphism, this is a full description of indecomposable objects in $\D$.) Denote by $M_{\alpha}$ the class of indecomposable corresponding to root $\alpha \in R_{+}$. Peng and Xiao then constructed a Lie algebra $\Hall_{\D}$ as follows:
\begin{enumerate}
\item Set $\Hall_{\D} = \mathfrak{N} \oplus \mathfrak{H}$ where $\mathfrak{H} = \mathcal{K} (\D)$ and $\mathfrak{N}$ is the free abelian group with basis $\{ u_{[M]} \}$ indexed by isomorphism classes of objects $[M]$. 
\item Let $h_{M} = [M] \in \mathcal{K} (\D)$.
\item Define a bilinear operation $[ \cdot , \cdot ]$ on $\Hall_{\D}$ by:
	\begin{enumerate}
	\item $[ \mathfrak{H} , \mathfrak{H} ] = 0$
	\item $[u_{M} , u_{N} ] = \sum_{[L]} (F_{M,N}^{L} (1) - F_{N,M}^{L} (1)) u_{L}$,  for $N \neq TM$, where $F_{M,N}^{L}$ is the Hall polynomial.
	\item $[ u_{M} , u_{TM} ] = \frac{h_{M}}{d(M)}$ where $d(M) = \dim \End (M)$
	\item $[ h_{M} , u_{N} ] = -( M, N )_{\D} u_{N} = -[ u_{N} , h_{M} ]$ where $( \cdot , \cdot )_{\D}$ is the symmetrized Euler form on $\mathcal{K} (\D)$. 
	\end{enumerate}
\item For $\alpha \in R_{+}$, let $h_{\alpha} = h_{M}$ where $\dim M_{\alpha} = \alpha$.
\item For $\alpha \in R_{+}$, let $M_{\alpha} = [ M_{\alpha} ]$ where $\dim M_{\alpha} = \alpha$.
\item For $\alpha \in R_{+}$ let $M_{-\alpha} = -[TM_{\alpha}]$ where $\dim M_{\alpha} = \alpha$.
\end{enumerate}

\begin{thm} Let $( \Hall_{\D} , [ \cdot , \cdot ] )$ be defined as above.
\begin{enumerate}
\item $( \Hall_{\D} , [ \cdot , \cdot ] )$ is a Lie algebra.
\item The collection $\{ M_{\alpha} , M_{-\alpha} \}_{\alpha \in R_{+}}$ defined above is a root basis for $\Hall_{\D}$.
\item The map given by $E_{\alpha} \mapsto M_{\alpha}$, $F_{\alpha} \mapsto M_{-\alpha}$ and $H_{\alpha} \mapsto h_{\alpha}$ for $\alpha \in R_{+}$ induces an isomorphism of Lie algebras $\g \to\Hall_{\D}$. Hence $\Hall_{\D}$ can be identified with the $\Z$-form of $\g$. 
\end{enumerate}
\end{thm}

For details see \cite{px} Section 4.

Recall that given a height function $h$, there is a corresponding set of simple roots $\Pi_{h}$ and a polarization $R= R_{+}^{h} \cup R_{-}^{h}$. Let $E_{\alpha}$ be the root vectors defined in Section~\ref{s:rootvectors}.  Define a triangular decomposition $\g = \n_{-}^{h} \oplus \h \oplus \n_{+}^{h}$ by setting $\n_{\pm}^{h} = \< E_{\alpha} \>_{\alpha \in R_{\pm}^{h}}$. 

A height function $h$ also gives an orientation $\Om_{h}$ of $\Gamma$ and hence a quiver $\overrightarrow{\Gamma} = (\Gamma , \Om_{h})$. As above, denote by $\mathcal{K}$ the corresponding Grothendieck group, and by $\Ind \subset \mathcal{K}$ the set of indecomposable classes in $\mathcal{K}$. Then there is a bijection  $R_{+}^{h} \to \Ind$, given by $\alpha \mapsto [M_{\alpha}]$. 

\begin{prop}\label{p:identification}
Let $h$ be a height function. Then the identification $R_{+}^{h} \to \Ind$ in $\Rep(\overrightarrow{\Gamma})$ induces an isomorphim $U \n_{+}^{h} \to \Hall_{1}$ given by $E_{\alpha} \mapsto [M_{\alpha}]$. 
\\
Moreover, the identification $R \to \Ind(\D)$ in the root category $\D$ gives an isomorphism $\g - \h \to \Hall_{\D} - \mathfrak{H}$, given by $E_{\alpha} \mapsto [M_{\alpha}]$, $E_{-\alpha} \mapsto -[TM_{\alpha}]$ for $\alpha \in R_{+}^{h}$.
\end{prop}

\begin{figure}[b]
\includegraphics[height=2.50in]{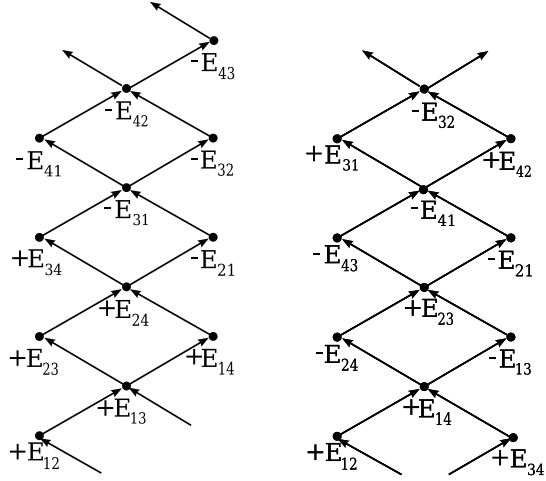}
\caption{Two different root bases for $\mathfrak{sl}_{4}$ coming from different choices of $C$. The case $C = (1234)$ is shown in the figure to the left. The case $C = (1243)$ is shown in the figure to the right. For each vertex in $\Gammahat$ the corresponding root vector $E_{\alpha}$ is shown in terms of the matrix units $E_{ij}$. Recall that $\Gammahat$ is periodic, so that arrows leaving the top level are identified with the incoming arrows on the bottom level.}\label{f:a3}
\end{figure}

\begin{cor}\label{c:bracket}
The Lie algebra $\g$ can be realised combinatorially in terms of $\Gammahat$: It has root basis $E_{\alpha}$ for $\alpha \in \Gammahat$ and Lie bracket given by 
\begin{equation}\label{e:bracket}
 [E_{\alpha} , E_{\beta} ] =\begin{cases}
(-1)^{ \< \alpha , \beta \>} E_{\alpha + \beta} &\text{ for } \alpha + \beta \in R \\
0 &\text{ for } \alpha + \beta \not \in R \text{ and } \alpha \neq -\beta
\end{cases}
\end{equation}
\end{cor}

\begin{proof}
The only thing to be checked is that in terms of the $E_{\alpha}$ constructed in Section~\ref{s:rootvectors}, the structure constants of the Lie bracket are given by Equation~\ref{e:bracket}. For $\alpha , \beta \in R$ with $\alpha \neq -\beta$ there is a choice of compatible simple roots $\Pi$ so that $\alpha , \beta \in R_{+}^{\Pi}$. Let $h$ be the corresponding height function. Then by Proposition~\ref{p:identification} the identification $U \n^{h}_{+} \simeq \Hall_{1}$ gives that $$[E_{\alpha} , E_{\beta} ] = [M_{\alpha} , M_{\beta} ]= (-1)^{\< \alpha , \beta \> } M_{\alpha + \beta} = (-1)^{\< \alpha , \beta \> } E_{\alpha + \beta}.$$
\end{proof}

\begin{remark} Note that the ``Euler cocylce" $(-1)^{ \< \cdot , \cdot \>}$, defines a cohomologous cocycle, and hence the same extension, as in the construction of $\g$ given in \cite{flm}.
\end{remark}

\begin{example}\label{e:a42}
Consider the case $\Gamma = A_{3}$, so that $\g = \mathfrak{sl}_{4}$. Let $\h$ be the diagonal matrices. Then the roots are $\alpha = e_{i} - e_{j}$ for $i \neq j$, where $e_{k} (h) = h_{kk}$ for $h \in \h$. The root space corresponding to root $e_{i} - e_{j}$ is $\CC E_{ij}$, where $E_{ij}$ is the corresponding matrix unit. For two different choices of Coxeter element $C$, two different root bases are shown in  Figure~\ref{f:a3}. In each case the Lie bracket is then given by the Equation~\ref{e:bracket} and the form $\< \cdot , \cdot \>$ can be computed explicitly in terms of $\Gammahat$.
\end{example}

%%%%%%%%%%%%%%%%%%%%%%%%%%%%%%%%%%%%%%%%%%%%%%%%%%%%%%%%%%%
\bibliographystyle{amsalpha}

\end{document}